\documentclass[11pt,reqno]{amsart}

\DeclareMathAlphabet\mathbfcal{OMS}{cmsy}{b}{n}
\usepackage{xcolor}
\usepackage{amssymb}
\usepackage{amscd}
\usepackage{amsfonts}
\usepackage{mathrsfs}
\usepackage{setspace}
\usepackage{version}
\usepackage{mathtools}
\usepackage{textcomp}
\usepackage{stmaryrd}
\usepackage[english]{babel}
\usepackage[colorlinks=true,linkcolor=blue]{hyperref}
\usepackage[nice]{nicefrac}

\newtheorem{theorem}{Theorem}[section]
\newtheorem{lemma}[theorem]{Lemma}

\theoremstyle{definition}

\theoremstyle{remark}
\newtheorem{remark}[theorem]{Remark}
\numberwithin{equation}{section}

\newcommand{\essinf}{{\rm ess.inf}}

\newcommand{\PP}{\mathbb{P}}

\begin{document}

	\title[Sharp large deviation estimates  for heavy-tailed extrema]{Sharp large deviation estimates  for heavy-tailed extrema}

	
		\author{Jos\'e M.~Zapata}
	\address{Universidad de Murcia.  Dpto. de Estadística e Investigación Operativa,  30100 Espinardo, Murcia, Spain}
	\email{jmzg1@um.es}
	
\thanks{The author acknowledges the partial support of the Ministerio de Ciencia e Innovación
of Spain in the project PID2022-137396NB-I00, funded by MICIU/AEI/10.13039/501100011033
and by 'ERDF A way of making Europe'.}

		\date{\today}

	\subjclass[2010]{}

\begin{abstract} 
We establish sharp large deviation asymptotics for the maximum order statistic of independent and identically distributed heavy-tailed random variables, valid for all Borel subsets of the right tail. This result yields exact  decay rates for exceedance probabilities at thresholds that grow faster than the natural extreme-value scaling. As an application, we derive the polynomial rate of decay of ruin probabilities in insurance portfolios where insolvency is driven by a single extreme claim.

   
   \mbox{} 
	\end{abstract}

	\maketitle
	
	\setcounter{tocdepth}{1}

\section{Main result}\label{sec:main}
The present work establishes a new large deviation principle in the context of extreme value theory and analyzes its implications for the asymptotic behavior of extreme insurance claims under heavy-tailed loss distributions. 
Let $(X_n)_{n\in\mathbb{N}}$ be independent and identically distributed
 random variables. We assume that the  distribution has a heavy right tail, in the sense
that its survival function $\bar F$ satisfies
\[
\bar F(x) := \mathbb{P}(X_1 > x) = x^{-\alpha} L(x), \qquad x > x_0,
\]
for some $x_0 > 0$, where $\alpha > 0$ and $L$ is a slowly varying function, i.e.,
\[
\lim_{x\to\infty} \frac{L(tx)}{L(x)} = 1,
\qquad \text{for all } t > 0.
\]
This assumption covers a wide range of claim size distributions commonly used
in actuarial science and finance, including the Pareto distribution, the Student-t distribution, the Burr distribution, stable distributions with index $\alpha<2$, and the log-gamma distribution; see
Embrechts et al.~\cite{Embrechts}.

Let $X_{(n)} := \max_{1 \le i \le n} X_i$ 
denote the maximum order statistic, which naturally arises in actuarial
applications such as the modeling of the largest individual claim, extreme portfolio losses,
or catastrophic risk over a fixed observation period.  

A classical result in extreme value theory asserts that there exists a sequence $(a_n)_{n\in\mathbb{N}}$ of positive normalizing constants such that
\[
\frac{X_{(n)}}{a_n}
\;\xrightarrow{d}\;
\Phi_{\alpha},
\]
where $\Phi_{\alpha}(x) = \exp(-x^{-\alpha})$, $x > 0$, is the Fr\'echet distribution with shape parameter $\alpha > 0$. A standard choice for the normalizing sequence $a_n$ is given by 
\begin{equation}\label{eq:an_def}
a_n := F^{\leftarrow}\!\left(1-\frac{1}{n}\right),
\qquad n \in \mathbb{N},
\end{equation}
where $F^{\leftarrow}(u) := \inf\{x \in \mathbb{R} : F(x) \ge u\}$ denotes the left-inverse of the distribution function $F$ of $X_1$. This result traces back to the pioneering work of Fisher and Tippett~\cite{fisher-tippett}
and Gnedenko~\cite{gnedenko}; see also~\cite{Embrechts,resnick} for comprehensive treatments
of heavy-tailed distributions and regular variation. 

While this result characterizes the typical scale of the maximum order statistic through weak convergence, it does not yield information on the  rate of decay of tail probabilities of the form
$$
\mathbb{P}(X_{(n)}>c_n a_n),
$$
when the threshold sequence $(c_n)_{n\in\mathbb{N}}$ increases with $n$. 
Such probabilities correspond to extreme  scenarios in which the maximum order statistic exceeds its natural extreme-value scale by a growing factor. We provide further illustration of this topic and its relevance for capital assessment in insurance in Section \ref{sec:ruin}.

To address this issue, we adopt a large deviation approach (see \cite{dembo}) tailored to the
maximum order statistic. Specifically,  we investigate the large deviations of the rescaled sequence
\begin{equation}\label{eq:normalized}
Z_n
=
\left(\frac{X_{(n)}}{a_n}\right)^{\frac{\alpha}{\log n}},
\end{equation}
where $a_n$ is defined by \eqref{eq:an_def}. 
 We further assume that $X_1$ admits a density $f$ which is non-increasing for sufficiently large values of $x$ (this assumption can actually be relaxed; see Remark~\ref{lem:conditions} below).

Our main result establishes a sharp large deviation principle for the right tail of
the sequence $(Z_n)_{n\in\mathbb{N}}$.

\begin{theorem}\label{thm:main}
Let $(Z_n)_{n\in\mathbb{N}}$ be defined by~\eqref{eq:normalized} under the assumptions above. Then, for every Borel set $A \subset [1, \infty)$,
\begin{equation}\label{eq:LDP}
\lim_{n\to\infty}
\frac{1}{\log n}\,
\log \PP(Z_n \in A)
=
- \essinf_{x \in A} \log x.
\end{equation}
\end{theorem}

Theorem~\ref{thm:main} shows that the sequence $(Z_n)_{n\in\mathbb{N}}$
satisfies a large deviation principle with speed $\log n$ and  rate
function $I(x)=\log x$ on $[1,\infty)$. The large deviation principle is \emph{sharp} in the sense that it yields an
exact limit for the logarithmic asymptotics of tail probabilities, rather than
upper and lower bounds as in the classical formulation of large deviation
theory; see, for example, \cite{dembo}. Moreover, the explicit form of the rate
function implies a universal decay behavior that depends only on the tail index
$\alpha$ and is insensitive to the slowly varying component of the underlying
distribution. 
 
The remainder of the paper is organized as follows. In Section~\ref{sec:ruin}, we study the implications of this result for estimating the decay of ruin probabilities for maximum claim portfolios.  Section~\ref{sec:proof} is devoted to the proof of
Theorem~\ref{thm:main}. The paper ends with a conclusion section.

\section{Decay of ruin probabilities for maximum claim portfolios}
\label{sec:ruin}

In solvency regulation and risk management, a key objective is to understand how ruin probabilities behave when capital requirements are increased to protect against extreme but plausible losses. This is particularly relevant for portfolios
exposed to heavy-tailed risks, where total losses may be dominated by a single
exceptionally large claim, as is commonly observed in catastrophe, liability,
and operational risk insurance; see, for example, \cite{Embrechts,mcneil}.

Consider an insurance portfolio consisting of $n$ independent policies, and let
$X_1, \ldots, X_n$ denote the corresponding claim sizes. The insurer charges a
premium $\pi_n$ per policy, so that the total available capital equals $n\pi_n$.
We focus on a ruin mechanism driven by a single catastrophic claim and define
ruin to occur if the largest claim exceeds the total premium income. The
corresponding ruin probability is
\[
\mathrm{RP}_n
=
\mathbb{P}(X_{(n)} > n\pi_n),
\qquad
X_{(n)} := \max_{1 \le i \le n} X_i .
\]

If the premium is chosen at the classical extreme value level
$\pi_n = a_n n^{-1}$, see \eqref{eq:an_def}, then the Fr\'echet limit implies
\[
\mathrm{RP}_n \longrightarrow 1-\Phi_\alpha(1) > 0,
\]
so that the probability of ruin does not vanish as the portfolio size grows.
This regime is therefore insufficient for ensuring asymptotic solvency.

To study the effect of more conservative capital requirements, we consider a
premium of the form
\[
\pi_n = a_n n^{\beta-1},
\qquad \beta > 0,
\]
which corresponds to a polynomial increase in capital relative to the standard
extreme value scaling. In this case, the ruin probability satisfies
\[
\mathrm{RP}_n
=
\mathbb{P}(X_{(n)} > a_n n^\beta)
\longrightarrow 0,
\qquad n \to \infty,
\]
ensuring asymptotic safety. However, classical extreme value theory does not quantify how fast this
probability converges to zero. In contrast, using Theorem~\ref{thm:main}, we obtain a precise
asymptotic rate. Indeed, noting that
\[
X_{(n)} > n\pi_n
\quad\Longleftrightarrow\quad
Z_n > e^{\alpha\beta},
\]
the large deviation principle~\eqref{eq:LDP} yields
\begin{equation}\label{eq:ruin_LDP}
\lim_{n\to\infty}
\frac{1}{\log n}
\log \mathrm{RP}_n
=
- \alpha \beta.
\end{equation}
Equivalently,
\begin{equation}\label{eq:ruin_poly}
\mathrm{RP}_n
=
n^{-\alpha\beta + o(1)},
\qquad n\to\infty.
\end{equation}

Expression~\eqref{eq:ruin_poly} shows that ruin probabilities decay at a
polynomial rate in the portfolio size. The tail index $\alpha$ determines the
sensitivity of solvency to capital stress, while the parameter $\beta$ controls
the degree of conservatism in premium loading.

From an actuarial perspective, this result offers a  tool for capital  assessment, since the large deviation approach yields the correct asymptotic order of ruin probabilities in extreme regimes.

\section{Proof of the main result}\label{sec:proof}
In this section, we provide the proof of the main result Theorem~\ref{thm:main}. 
Let $(X_n)_{n\ge1}$ be i.i.d.\ random variables satisfying the assumptions in Section \ref{sec:main}, and let $X_{(n)}=\max_{1\le i\le n} X_i$. 
Let $Z_n$ be defined by~\eqref{eq:normalized}. Introduce the notation
\[
t_n(x)=a_n x^{\frac{\log n}{\alpha}},
\qquad
I(x)=\log x.
\]

Let $G_n$ and $g_n$ denote the distribution function and density of $Z_n$, respectively. A direct computation yields
\begin{equation}\label{eq:distribution}
G_n(x)=F^n\bigl(t_n(x)\bigr),
\qquad
g_n(x)=n\,a_n\,\frac{\alpha}{\log n}\,
x^{\frac{\log n}{\alpha}-1}\,
F^{\,n-1}\!\bigl(t_n(x)\bigr)\,
f\!\bigl(t_n(x)\bigr).
\end{equation}

The following lemma collects limits related to the normalizing sequence
$(a_n)_{n\in\mathbb{N}}$ that will be used in the proof of the main result.

\begin{lemma}\label{lem:an}
The following limits hold:
\begin{equation}\label{eq:an-log}
\lim_{n\to\infty}\frac{\log a_n}{\log n}=\frac{1}{\alpha}.
\end{equation}
\begin{equation}\label{eq:fan-log}
\lim_{n\to\infty}\frac{\log \bar{F}(a_n)}{\log n}=-1.
\end{equation}
\end{lemma}

\begin{proof}
By \cite[Proposition~0.8(ii)]{resnick}, the tail distribution satisfies
\begin{equation}\label{eq:log-tail}
\lim_{x\to\infty}\frac{\log \bar F(x)}{\log x} = -\alpha.
\end{equation}

We first note that $a_n\to\infty$. Indeed, if $(a_n)$ were bounded, we would have $\bar F(x)=0$ for all
sufficiently large $x$, implying that the distribution has bounded support.
This contradicts \eqref{eq:log-tail}.

Fix $\varepsilon>0$. By \eqref{eq:log-tail}, there exists $x_\varepsilon>0$ such
that for all $x\ge x_\varepsilon$,
\begin{equation}\label{eq:power-bounds}
x^{-(\alpha+\varepsilon)}
\;\le\;
\bar F(x)
\;\le\;
x^{-(\alpha-\varepsilon)}.
\end{equation}

Since $a_n\to\infty$, for $n$ large enough we have $a_n\ge x_\varepsilon$.
By the definition of $a_n$,
\begin{equation}\label{eq:quantile-def}
\bar F(a_n)\le \frac{1}{n}.
\end{equation}
Using the lower bound in \eqref{eq:power-bounds}, we obtain
\[
a_n^{-(\alpha+\varepsilon)} \le \frac{1}{n},
\]
which implies
\begin{equation}\label{eq:lower-bound}
a_n \ge n^{\frac{1}{\alpha+\varepsilon}}.
\end{equation}

Moreover, since $\bar F$ is non-increasing and $a_n$ is defined as an infimum,
we have for every fixed $\varepsilon>0$,
\begin{equation}\label{eq:before-quantile}
\bar F(a_n-\varepsilon)\ge \frac{1}{n},
\end{equation}
for all $n$ sufficiently large. For such $n$, we also have
$a_n-\varepsilon\ge x_\varepsilon$, and applying the upper bound in
\eqref{eq:power-bounds} yields
\[
\frac{1}{n}
\;\le\;
(a_n-\varepsilon)^{-(\alpha-\varepsilon)}.
\]
Hence,
\begin{equation}\label{eq:upper-bound}
a_n
\;\le\;
\varepsilon + n^{\frac{1}{\alpha-\varepsilon}}.
\end{equation}

Combining \eqref{eq:lower-bound} and \eqref{eq:upper-bound}, we obtain
\[
\frac{1}{\alpha+\varepsilon}
\;\le\;
\frac{\log a_n}{\log n}
\;\le\;
\frac{1}{\alpha-\varepsilon}
+
\frac{\log\!\left(1+\varepsilon n^{-\frac{1}{\alpha-\varepsilon}}\right)}{\log n}.
\]
The second term on the right converges to zero as $n\to\infty$. Therefore,
\[
\frac{1}{\alpha+\varepsilon}
\le
\liminf_{n\to\infty}\frac{\log a_n}{\log n}
\le
\limsup_{n\to\infty}\frac{\log a_n}{\log n}
\le
\frac{1}{\alpha-\varepsilon}.
\]
Letting $\varepsilon\downarrow0$ proves \eqref{eq:an-log}.

Finally, it follows from \eqref{eq:log-tail} that
\[
\lim_{n\to\infty}\frac{\log\bar{F}(a_n)}{\log a_n}=-\alpha.
\]
Consequently,
\begin{align*}
\lim_{n\to\infty}\frac{\log\bar{F}(a_n)}{\log n}
&=
\lim_{n\to\infty}
\frac{\log\bar{F}(a_n)}{\log a_n}
\cdot
\frac{\log a_n}{\log n} \\
&=
(-\alpha)\cdot \frac{1}{\alpha}
=
-1,
\end{align*}
where we have used \eqref{eq:an-log}. This proves \eqref{eq:fan-log}.
\end{proof}

The following result is classical in the theory of regularly varying functions and is commonly referred to as the \emph{Potter bounds}; see, for instance, 
\cite[Proposition~0.8(ii), p.~22]{resnick} or \cite[Theorem~1.5.6]{bingham}.

\begin{lemma}\label{lem:Potter} 
For every $\varepsilon>0$ there exists $t_0>0$ such that
\[
(1-\varepsilon)\,x^{-\alpha-\varepsilon}
\;\le\;
\frac{\bar F(tx)}{\bar F(t)}
\;\le\;
(1+\varepsilon)\,x^{-\alpha+\varepsilon},
\]
for all $t\ge t_0$ and all $x\ge x_0$.
\end{lemma}

We first establish an upper bound for the right tail probabilities, which corresponds
to the large-deviation upper bound in Theorem~\ref{thm:main}.

\begin{lemma}\label{lem:limsup}
For every \(x \ge 1\),
\[
\limsup_{n\to\infty}\frac{1}{\log n}\log \PP(Z_n>x)\le -I(x).
\]
Recall that $I(x)=\log x$.
\end{lemma}

\begin{proof}
Fix \(x\ge1\).
Using the elementary identity
\[
1-u^n=(1+u+\cdots+u^{n-1})(1-u), \qquad u\in[0,1],
\]
we obtain the inequality
\[
1-u^n \le n(1-u), \qquad u\in[0,1].
\]
Applying this with \(u=F(t_n(x))\), we obtain
\[
\PP(Z_n>x)
=
1-F(t_n(x))^n
\le
n\,\bar F(t_n(x)).
\]

Since \(x\ge1\), we have \(t_n(x)\ge a_n\to\infty\) as \(n\to\infty\).
Fix \(\varepsilon>0\). By the Potter bounds (Lemma~\ref{lem:Potter}), for all \(n\) sufficiently large,
\[
\bar F(t_n(x))
\le
(1+\varepsilon)\,\bar F(a_n)\,x^{\frac{\log n}{\alpha}(-\alpha+\varepsilon)}.
\]
Since \(\bar F(a_n)\le 1/n\), we deduce
\[
\PP(Z_n>x)
\le
(1+\varepsilon)\,x^{\frac{\log n}{\alpha}(-\alpha+\varepsilon)}.
\]

Taking logarithms, dividing by \(\log n\) , and taking the limit superior yield 
\[
\limsup_{n\to\infty}
\frac{1}{\log n}\log \PP(Z_n>x)
\le
\left(-1+\frac{\varepsilon}{\alpha}\right)\log x.
\]
Letting \(\varepsilon\downarrow0\) completes the proof.
\end{proof}

\begin{lemma}\label{lem:uniform}
The density $g_n$ of $Z_n$ satisfies
\[
\lim_{n\to\infty}\frac{1}{\log n}\log g_n(x)=-I(x),
\]
with convergence uniform for $x\in[1,M]$, for any fixed $M>1$.
\end{lemma}

\begin{proof}
Recall that the density $f$ of $X_1$ is assumed to be non-increasing for $x$
large enough. Under this assumption, the Von Mises condition holds:
\begin{equation}\label{eq:VonMises}
\lim_{x\to\infty}\frac{x f(x)}{\bar F(x)}=\alpha,
\end{equation}
see \cite[Proposition~1.15(b)]{resnick}. This asymptotic relation will be used
 below.

From \eqref{eq:distribution}, we may write
\[
\frac{1}{\log n}\log g_n(x)
=
T_1(n)+T_2(n,x)+T_3(n,x)+T_4(n,x),
\]
where
\[
T_1(n)
=\frac{1}{\log n}\log\!\left(n\,a_n\,\frac{\alpha}{\log n}\right),\quad
T_2(n,x)
=\left(\frac{1}{\alpha}-\frac{1}{\log n}\right)\log x,
\]
\[
T_3(n,x)
=\frac{n-1}{\log n}\log F\!\bigl(t_n(x)\bigr),\quad
T_4(n,x)
=\frac{1}{\log n}\log f\!\bigl(t_n(x)\bigr).
\]
We analyze each term separately.

By Lemma~\ref{lem:an}, $\log a_n=(1/\alpha+o(1))\log n$, which immediately yields
\[
\lim_{n\to\infty}T_1(n)=1+\frac{1}{\alpha}.
\]

Moreover,
\[
\lim_{n\to\infty}T_2(n,x)=\frac{1}{\alpha}\log x,
\]
with convergence uniform for $x\in[1,M]$.

Next, since $x\ge1$, we have $t_n(x)\ge a_n\to\infty$. As $F$ is non-decreasing,
\[
0\le -\log F(t_n(x))\le -\log F(a_n).
\]
For $u$ sufficiently close to $1$, the inequality $0\le -\log u\le 2(1-u)$ holds.
Using $\bar F(a_n)\le 1/n$, we obtain
\[
-\log F(a_n)\le 2\bar F(a_n)\le \frac{2}{n}.
\]
Consequently,
\[
|T_3(n,x)|
\le
\frac{n-1}{\log n}\frac{2}{n},
\]
which converges to $0$ as $n\to\infty$, uniformly for $x\in[1,M]$.

We now turn to the fourth term,
\[
T_4(n,x)=\frac{1}{\log n}\log f\bigl(t_n(x)\bigr).
\]
Since $f$ is eventually non-increasing, the Von Mises condition
\eqref{eq:VonMises} implies that, for $n$ large enough and all $x\in[1,M]$,
\begin{equation}\label{eq:vm-bounds}
(\alpha-1)\frac{\bar F(t_n(x))}{t_n(x)}
\le
f(t_n(x))
\le
(\alpha+1)\frac{\bar F(t_n(x))}{t_n(x)}.
\end{equation}

Applying iteratively the Potter bounds, we may construct a sequence
$\varepsilon_n\downarrow0$ such that, for all $n$ sufficiently large and all
$x\in[1,M]$,
\begin{equation}\label{eq:potter-tn}
(1-\varepsilon_n)\,
x^{-\frac{\log n}{\alpha}(\alpha+\varepsilon_n)}\,\bar{F}(a_n)
\le \bar{F}(t_n(x))
\le (1+\varepsilon_n)\,
x^{-\frac{\log n}{\alpha}(\alpha-\varepsilon_n)}\,\bar{F}(a_n).
\end{equation}

Combining the upper bounds in \eqref{eq:vm-bounds} and \eqref{eq:potter-tn}, we
obtain
\begin{align*}
\frac{1}{\log n}\log f\bigl(t_n(x)\bigr)
&\le
\frac{\log(\alpha+1)}{\log n}
-\frac{\log a_n}{\log n}
-\frac{1}{\alpha}\log x \\
&\quad
+\frac{1}{\log n}\log\!\left(\frac{1+\varepsilon_n}{n}\right)
-\left(1-\frac{\varepsilon_n}{\alpha}\right)\log x
+\frac{1}{\log n}\log\bar F(a_n).
\end{align*}
Using Lemma~\ref{lem:an}, we deduce that, uniformly for $x\in[1,M]$,
\[
\frac{1}{\log n}\log f\bigl(t_n(x)\bigr)
\le
-\left(1+\frac{1}{\alpha}\right)
-\left(1+\frac{1}{\alpha}\right)\log x
+o(1),
\]
where the $o(1)$ term does not depend on $x$.

Similarly, combining the lower bounds in \eqref{eq:vm-bounds} and
\eqref{eq:potter-tn}, we obtain that, uniformly for $x\in[1,M]$,
\[
\frac{1}{\log n}\log f\bigl(t_n(x)\bigr)
\ge
-\left(1+\frac{1}{\alpha}\right)
-\left(1+\frac{1}{\alpha}\right)\log x
+o(1).
\]
Together, these bounds imply
\[
\lim_{n\to\infty}T_4(n,x)
=
-\left(1+\frac{1}{\alpha}\right)
-\left(1+\frac{1}{\alpha}\right)\log x,
\]
uniformly for $x\in[1,M]$.

Collecting the limits of $T_1$, $T_2$, $T_3$, and $T_4$, we observe that the
constant terms cancel and the remaining contributions combine to give
\[
\lim_{n\to\infty}\frac{1}{\log n}\log g_n(x)
=
-\log x
=
-I(x),
\]
uniformly for $x\in[1,M]$. This completes the proof.
\end{proof}

We finally prove Theorem \ref{thm:main}.

\begin{proof}
Fix a Borel set $A\subset[1,\infty)$ with positive Lebesgue measure; otherwise the statement is trivial.  
Set
\[
I_A:=\essinf_{x\in A} I(x).
\]

Fix $\varepsilon>0$ and define
\[
A_\varepsilon:=\{x\in A:\ I(x)<I_A+\varepsilon\}.
\]
By definition of the essential infimum, $A_\varepsilon$ has positive Lebesgue measure.  
Choose $M>0$ large enough so that $A_\varepsilon\cap[0,M]$ also has positive Lebesgue measure.

By the uniform convergence established in Lemma~\ref{lem:uniform}, there exists $n_0\in\mathbb{N}$ such that
\begin{equation}\label{eq:uniform-density}
n^{-I(x)-\varepsilon}
\le
g_n(x)
\le
n^{-I(x)+\varepsilon},
\qquad
\text{for all } x\in[0,M],\ n\ge n_0.
\end{equation}

We now derive matching lower and upper bounds for the logarithmic asymptotics.

First, for the lower bound, we obtain
\begin{align*}
\liminf_{n\to\infty}\frac{1}{\log n}\log\PP(Z_n\in A)
&\ge
\liminf_{n\to\infty}\frac{1}{\log n}\log\PP(Z_n\in A_\varepsilon\cap[0,M])\\
&=
\liminf_{n\to\infty}\frac{1}{\log n}
\log\int_{A_\varepsilon\cap[0,M]} g_n(x)\,dx\\
&\ge
\liminf_{n\to\infty}\frac{1}{\log n}
\log\int_{A_\varepsilon\cap[0,M]} n^{-I(x)-\varepsilon}\,dx\\
&\ge
\liminf_{n\to\infty}\frac{1}{\log n}
\log\int_{A_\varepsilon\cap[0,M]} n^{-I_A-2\varepsilon}\,dx\\
&=
-I_A-2\varepsilon,
\end{align*}
where we used that $I(x)\le I_A+\varepsilon$ on $A_\varepsilon$ and that the Lebesgue measure of $A_\varepsilon\cap[0,M]$ is positive.

Next, for the upper bound, we decompose
\[
\PP(Z_n\in A)
\le
\PP(Z_n\in A\cap[0,M])+\PP(Z_n>M)
\le
2\max\{\PP(Z_n\in A\cap[0,M]),\PP(Z_n>M)\}.
\]
Taking logarithms and dividing by $\log n$, we obtain
\begin{align*}
\limsup_{n\to\infty}\frac{1}{\log n}\log\PP(Z_n\in A)
\le
\max\Bigg\{
&\limsup_{n\to\infty}\frac{1}{\log n}\log\PP(Z_n\in A\cap[0,M]),\\
&\limsup_{n\to\infty}\frac{1}{\log n}\log\PP(Z_n>M)
\Bigg\}.
\end{align*}

By Lemma~\ref{lem:limsup}, we have
\[
\limsup_{n\to\infty}\frac{1}{\log n}\log\PP(Z_n>M)\le -I(M).
\]
On the other hand, using the upper bound in \eqref{eq:uniform-density},
\begin{align*}
\limsup_{n\to\infty}\frac{1}{\log n}\log\PP(Z_n\in A\cap[0,M])
&\le
\limsup_{n\to\infty}\frac{1}{\log n}
\log\int_{A\cap[0,M]} n^{-I(x)+\varepsilon}\,dx\\
&\le
-I_A+\varepsilon.
\end{align*}

Combining the lower and upper bounds yields
\[
-I_A-2\varepsilon
\le
\liminf_{n\to\infty}\frac{1}{\log n}\log\PP(Z_n\in A)
\le
\limsup_{n\to\infty}\frac{1}{\log n}\log\PP(Z_n\in A)
\le
\max\{-I_A+\varepsilon,\,-I(M)\}.
\]

Finally, letting $M\uparrow\infty$ and then $\varepsilon\downarrow0$, we obtain
\[
\lim_{n\to\infty}\frac{1}{\log n}\log\PP(Z_n\in A)
=
-I_A,
\]
which completes the proof.
\end{proof}

We discuss the assumptions of Theorem \ref{thm:main}.
\begin{remark}\label{lem:conditions}
While  the von Mises condition \eqref{eq:VonMises} holds whenever $f$ is eventually monotone, the proof of Theorem \ref{thm:main} only requires the weaker assumption that the ratio
\begin{equation}\label{eq:ratio}
\left|\frac{x f(x)}{\bar{F}(x)}\right|
\end{equation}
is bounded for sufficiently large $x$.

By \cite[Corollary 1.12]{resnick}, $\bar{F}$ has the representation
$$
\bar{F}(x)=c(x)\exp\left(-\int_ 1^x\frac{\alpha(t)}{t} dt\right).
$$ 
for $x\ge 1$, where $\lim_{x\to\infty}c(x)=c>0$ and $\lim_{x\to\infty}\alpha(x)=\alpha>0$. 

By differentiating the above representation, we get
$$
\frac{x f(x)}{\bar{F}(x)}=\alpha(x)-\frac{xc^\prime(x)}{c(x)}.
$$
Since $c(x)$, $\alpha(x)$ converge to positive constants, they are  bounded for large $x$.
 Consequently, to ensure the boundedness of \eqref{eq:ratio}, it suffices to assume that $|x c'(x)|$ is bounded for $x$ large enough, or equivalently, $c'(x) = O(1/x)$. So the monotonicity of $f$ (or the Von Mises condition) can be replaced by this condition in the proof of Theorem \ref{thm:main} to hold. Note that in the common case where $c(x)$ is constant for large $x$ (so $c'(x)=0$), the ratio simply converges to $\alpha$, recovering the classical von Mises condition \eqref{eq:VonMises}.
\end{remark}

\section{Conclusion}

In this paper, we established a novel sharp large deviation principle for the maximum
of i.i.d.\ heavy-tailed random variables. By introducing a logarithmic
rescaling of the classical extreme-value normalization, we derived tail
probability estimates that are valid on all Borel subsets of the extreme right
tail. The resulting large deviation principle has speed $\log n$ and an
explicit rate function, leading to polynomial decay of extreme tail
probabilities in the sample size $n$.

From an actuarial and risk-management perspective,  the
derived asymptotics yield explicit decay rates for ruin probabilities in models
where insolvency is driven by a single catastrophic claim. The resulting
polynomial rates depend only on the tail index $\alpha$ and are insensitive to
the slowly varying component of the claim size distribution, which is often
difficult to estimate reliably in practice.

The analysis highlights the limitations of classical extreme value
approximations when capital levels increase above their natural
scaling. While the Fr\'echet limit accurately describes typical extremes, it
does not capture the speed at which ruin probabilities vanish under
increasingly conservative capital requirements. The large deviation framework
developed in this paper fills this gap by providing asymptotic estimates that
remain informative in rare-event regimes relevant for stress testing and
capital adequacy assessment.

%

\end{document}